\documentclass[reqno]{amsart}
\usepackage{amsmath, amsthm, amscd, amssymb, amsfonts, amsbsy}
\usepackage{latexsym, color, enumerate}
\usepackage{pxfonts}
\usepackage[dvipsnames]{xcolor}
\usepackage{bbm}
\usepackage{mathrsfs}
\theoremstyle{plain}
\newtheorem{theorem}[equation]{Theorem}
\newtheorem{lemma}[equation]{Lemma}

\theoremstyle{definition}
\newtheorem{definition}[equation]{Definition}

\theoremstyle{remark}

\newcommand{\dv}{\operatorname{div}}

\newcommand{\tr}{\operatorname{tr}}

\numberwithin{equation}{section}

\newcommand{\bR}{\mathbb{R}}

\providecommand{\set}[1]{\{#1\}}
\providecommand{\Set}[1]{\left\{#1\right\}}

\providecommand{\abs}[1]{\lvert#1\rvert}
\providecommand{\Abs}[1]{\left\lvert#1\right\rvert}

\providecommand{\norm}[1]{\lVert#1\rVert}

\begin{document}
\title[PDEs in double divergence form]
{Recent progress on second-order elliptic and parabolic equations in double divergence form}

\author[S. Kim]{Seick Kim}
\address[S. Kim]{Department of Mathematics, Yonsei University, 50 Yonsei-ro, Seodaemun-gu, Seoul 03722, Republic of Korea}
\email{kimseick@yonsei.ac.kr}
\thanks{S. Kim is supported by the National Research Foundation of Korea (NRF) under agreement NRF-2022R1A2C1003322.}



\begin{abstract}
This article presents a comprehensive overview and supplement to recent developments in second-order elliptic partial differential equations formulated in double divergence form, along with an exploration of their parabolic counterparts.
\end{abstract}
\maketitle

\section{Introduction}

This article serves as a survey and supplements of recent advancements in second-order elliptic partial differential equations formulated in double divergence form. 
Furthermore, we will delve into their parabolic counterparts.

We examine the elliptic operator $L^*$ defined as
\[
L^*u=\dv^2(\mathbf A u)-\dv(\boldsymbol{b}u)=D_{ij}(a^{ij}u)-D_i(b^i u),
\]
operating on a domain $\Omega \subset \mathbb{R}^d$ with $d \ge 2$.
Here, the conventional summation convention over repeated indices is applied.
We assume that  $\mathbf A=(a^{ij})$ represents a $d \times d$ symmetric matrix valued function defined across the entire space $\bR^d$satisfying the ellipticity condition:
\begin{equation}					\label{ellipticity-nd}
a^{ij}=a^{ji},\quad \delta \abs{\xi}^2 \le a^{ij}(x) \xi_i \xi_j \le \delta^{-1} \abs{\xi}^2,\quad \forall x \in \Omega,\quad \forall \xi \in \mathbb{R}^d,
\end{equation}
where $\delta \in (0,1]$ is a constant.
Later, we will impose further assumptions on the coefficients $\mathbf A$ and $\boldsymbol b$.

We note that the operator $L^*$ is the formal adjoint of the elliptic operator $L$, where
\[
L v=\tr(\mathbf A D^2 v)+ \boldsymbol b \cdot Dv=a^{ij}D_{ij}v + b^i D_i v.
\]
An important example of a double divergence form elliptic equation is the stationary Kolmogorov equation for invariant measures of a diffusion process.
Interested readers are encouraged to refer to an excellent survey book on this subject \cite{BKRS15}.

We say that a function $g$ defined on $\Omega$ is of Dini mean oscillation and write $g \in \mathrm{DMO}$ if the mean oscillation function $\omega_g: \bR_+ \to \bR$ defined by
\begin{equation*}					
\omega_g(r):=\sup_{x\in \Omega} \fint_{\Omega \cap B_r(x)} \,\abs{g(y)-(g)_{\Omega \cap B_r(x)}}\,dy, \;\; \text{where }(g)_{\Omega \cap B_r(x)}:=\fint_{\Omega \cap B_r(x)} g,
\end{equation*}
satisfies the Dini condition, i.e.,
\[
\int_0^1 \frac{\omega_g(t)}t \,dt <+\infty.
\]
It is evident that if $g$ is Dini continuous, then $g$ is of Dini mean oscillation.
However, it is noteworthy that the Dini mean oscillation condition is distinctly less restrictive than Dini continuity; refer to \cite[p. 418]{DK17} for a concrete example.
Moreover, if $g$ is of Dini mean oscillation, it follows that $g$ is uniformly continuous, with its modulus of continuity being governed by $\omega_g$ as detailed in the Appendix of \cite{HK20}.

It has long been recognized that weak solutions to the double divergence form equation $\dv^2(\mathbf A u)=0$ are not necessarily continuous, even if  $\mathbf A$ is uniformly continuous and elliptic (see \cite{Bauman84b}).
However, if $\mathbf A$ is Dini continuous, then a weak solution of $\dv^2(\mathbf A u)=0$ is continuous.
More precisely, a weak solution has a continuous representative. In particular, if $\mathbf A$ is H\"older continuous, then a solution is H\"older continuous (see \cite{Sjogren73, Sjogren75}).

In recent papers \cite{DK17, DEK18}, it has been shown that weak solutions to the double divergence form equation, $\dv^2(\mathbf A u)=0$, are continuous, accompanied by estimates on their modulus of continuity, provided that $\mathbf A \in \mathrm{DMO}$.
This outcome has facilitated the development of the Green's function for the non-divergence form elliptic operator (with $b^i=c=0$) in $C^{1,1}$ domains, enabling the establishment of pointwise bounds comparable to those of the Green's function for Laplace's equation. See \cite{HK20, DK21, KL21, DKL23} and refer to \cite{DEK21,DKL22} for the parabolic counter part.

The organization of the paper is as follows:
In Section \ref{sec2}, we review elliptic equations in double divergence form, discussing existence and regularity of weak solutions.
In Section \ref{sec3}, we extend the result in Section \ref{sec2} to the case when the lower order coefficients belong to certain Morrey classes.
Parabolic counterparts are presented in Section \ref{sec4}.

\section{Elliptic equations in double divergence form}			\label{sec2}

\subsection{Definition of weak solutions}

Let $\Omega$ be a bounded $C^{1,1}$ domain and let $\nu$ denote the outer unit normal vector to $\partial \Omega$.

We consider the equation $L^*u-cu=\dv^2 \mathbf f + \dv \boldsymbol{g} + h$, where $\mathbf f=(f^{ij})_{i,j=1}^d$, $\boldsymbol{g}=(g^1,\ldots, g^d)$, and $h$ is a scalar function.
For $\eta \in C^\infty(\overline \Omega)$ such that $\eta=0$ on $\partial \Omega$, assuming that we can perform integration by parts to obtain:
\[
\int_\Omega \left(D_{ij}(a^{ij} u)-D_i(b^i u)-cu\right)\eta
=-\int_{\partial\Omega} a^{ij}u \nu_i D_j \eta + \int_\Omega u\left(a^{ij} D_{ij}\eta+ b^i D_i \eta- c \eta\right).
\]
Also, we can formally consider:
\[
\int_\Omega \left(D_{ij}f^{ij}+D_i g^i +h\right) \eta=-\int_{\partial\Omega} f^{ij} \nu_i D_j \eta + \int_\Omega f^{ij}D_{ij}\eta -g^i D_i \eta+h\eta.
\]

Note that it follows from the assumption $\eta=0$ on $\partial \Omega$ that $D\eta = (D\eta \cdot \nu) \nu$, and thus we have $D_j \eta=(D\eta \cdot \nu) \nu_j$.
If we wish the boundary integrals to disappear, for any $\eta \in C^\infty(\overline\Omega)$ satisfying $\eta=0$ on $\partial \Omega$, we should have
\[
0=a^{ij}u \nu_i D_j \eta-f^{ij}\nu_i D_j \eta=(a^{ij}u \nu_i\nu_j -f^{ij} \nu_i \nu_j)(D\eta \cdot \nu)\;\text{ on }\;\partial\Omega.
\]
Therefore, the natural boundary condition we seek is $a^{ij}\nu_i \nu_j u= f^{ij}\nu_i \nu_j$, which can be expressed as follows: 
\[
(\mathbf A \nu\cdot \nu) u= \mathbf f \nu \cdot \nu.
\]
Note that the uniform ellipticity condition implies that $\mathbf A \nu\cdot \nu \ge \delta >0$.
Therefore, we can consider the following more general boundary condition:
\[
u=\frac{\mathbf f \nu \cdot \nu}{\mathbf A \nu\cdot \nu}+\psi\;\text{ on }\;\partial \Omega.
\]
This means that 
\begin{align*}
a^{ij}u \nu_i D_j \eta=a^{ij}u \nu_i \nu_j (D \eta\cdot \nu)&= f^{ij}\nu_i \nu_j (D\eta \cdot \nu)+a^{ij} \nu_i \nu_j \psi(D \eta\cdot \nu)\\
&=f^{ij} \nu_i D_j\eta +a^{ij}D_j\eta \nu_i \psi \;\text{ on }\;\partial\Omega.
\end{align*}

\begin{definition}
Let $\Omega$ be a bounded $C^{1,1}$ domain in $\mathbb{R}^d$.
Assume that $\mathbf f=(f^{ij})$, $\boldsymbol{g}=(g^1,\ldots, g^d)$, $h \in L_1(\Omega)$, and $\psi \in L_1(\partial \Omega)$.
We say that $u\in L_1(\Omega)$ is a weak solution of
\[
L^*u-cu=\dv^2 \mathbf{f}+\dv \boldsymbol{g}+h \;\text{ in }\;\Omega,\quad u=\frac{\mathbf f \nu \cdot \nu}{\mathbf A \nu\cdot \nu}+\psi\;\text{ on }\;\partial\Omega,
\]
if $\boldsymbol{b} u \in L_1(\Omega)$, $cu \in L_1(\Omega)$, and for every $\eta \in C^\infty(\overline \Omega)$ satisfying $\eta=0$ on $\partial \Omega$, we have
\begin{equation}				\label{eq0958sat}
\int_\Omega u\left(a^{ij}D_{ij}\eta + b^i D_i \eta -c\eta\right) =\int_\Omega f^{ij}D_{ij}\eta - g^i D_i \eta+h \eta +\int_{\partial\Omega} \psi a^{ij}D_j\eta \nu_i.
\end{equation}
Let $U \subset \mathbb{R}^d$ be an open set, and let $\mathbf f$, $\boldsymbol g$, $h \in L_{1, \rm{loc}}(U)$.
We say that $u \in L_{1,\rm{loc}}(U)$ is a weak solution of
\[
L^*u-cu=\dv^2 \mathbf f+\dv \boldsymbol g +h\;\text{ in }\;U,
\]
if $\boldsymbol b u \in L_{1,\rm{loc}}(U)$, $cu \in L_{1,\rm{loc}}(U)$, and for $\eta \in C^\infty_c(U)$, we have
\[
\int_{U} u \left(a^{ij}D_{ij}\eta + b^i D_i \eta - c\eta\right) =\int_{U} f^{ij}D_{ij}\eta - g^i D_i \eta+h \eta.
\]
\end{definition}

\subsection{Existence of weak solutions}			\label{sec2.2}
Let $\Omega$ be a bounded $C^{1,1}$ domain in $\mathbb{R}^d$.
For simplicity of presentation, we assume that $d \ge 3$.
We seek a weak solution $u$ from the space $L_{q_0}(\Omega)$ satisfying \eqref{eq0958sat}, where $q_0 >d/(d-2)$.
This result follows from the so-called transposition or duality method, which relies on the existence and uniqueness of $\mathring W^{2}_{q_0'}(\Omega)$ solutions to
\[
L v - cv =\varphi \;\text{ in }\;\Omega,\quad v=0\;\text{ on }\;\partial\Omega.
\]
Here, $p'$ denotes the H\"older conjugate of $p$ and $\mathring W^{2}_{p}(\Omega)$ denotes the completion of the set $\set{u \in C^\infty(\overline \Omega): u=0 \;\text{ on }\;\partial\Omega}$ in  the norm of $W^2_p$.

We assume that $\boldsymbol b \in L_{p_0}(\Omega)$, $c \in L_{p_0/2}(\Omega)$ with $p_0>d$, and $c \ge 0$. 
For $q_0' \in (1,d/2)$, let $\mathbf f \in L_{q_0}(\Omega)$, $\boldsymbol g \in L_{d q_0/(d+q_0)}(\Omega)$,  $h \in L_{d q_0/(d+2 q_0)}(\Omega)$, and $\psi \in L_{q_0}(\partial\Omega)$.

Consider the mapping $T:L_{q_0'}(\Omega)\to \mathbb{R}$ given by
\[
T(\varphi)=\int_\Omega f^{ij}D_{ij} v- \int_\Omega g^i D_i v + \int_\Omega hv+\int_{\partial\Omega} \psi a^{ij}D_jv\nu_i,
\]
where $v \in \mathring W^2_{q_0'}(\Omega)$ is the strong solution of $Lv -c v =\varphi$.
We refer to \cite[Theorem 2.8]{Krylov2023e} for the unique solvability of $v$.
We note that Assumption 2.6 in \cite{Krylov2023e} is satisfied with $q_b=p_0$, $q_c=p_0/2$, and $p=q_0'$.
Moreover, we observe that 
\begin{align*}
\Abs{\int_\Omega f^{ij}D_{ij} v\,} &\le \norm{\mathbf f}_{L_{q_0}(\Omega)} \norm{D^2 v}_{L_{q_0'}(\Omega)}\le N \norm{\mathbf f}_{L_{q_0}(\Omega)} \norm{\varphi}_{L_{q_0'}(\Omega)},\\
\Abs{\int_\Omega g^{i}D_{i} v\,} &\le \norm{\boldsymbol g}_{L_{dq_0/(d+q_0)}(\Omega)} \norm{D v}_{L_{q_0'd/(d-q_0')}(\Omega)}\le N \norm{\boldsymbol g}_{L_{d q_0/(d+q_0)}(\Omega)} \norm{\varphi}_{L_{q_0'}(\Omega)},\\
\Abs{\int_\Omega h v\,} &\le \norm{h}_{L_{d q_0/(d+2q_0)}(\Omega)} \norm{v}_{L_{q_0'd/(d-2q_0')}(\Omega)}\le N \norm{h}_{L_{d q_0/(d+2q_0)}(\Omega)} \norm{\varphi}_{L_{q_0'}(\Omega)},\\
\Abs{\int_{\partial\Omega} \psi  a^{ij}D_jv \nu_i\,} &\le N  \norm{\psi}_{L_{q_0}(\partial\Omega)} \norm{Dv}_{W^1_{q_0'}(\Omega)}\le N \norm{\psi}_{L_{q_0}(\partial\Omega)} \norm{\varphi}_{L_{q_0'}(\Omega)}.
\end{align*}
Here, we used the trace theorem in the last inequality.
Therefore, $T$ is a bounded functional on $L_{q_0'}(\Omega)$, and by the Riesz representation theorem, there is unique $u \in L_{q_0}(\Omega)$ such that
\[
T(\varphi)=\int_\Omega u \varphi,\quad \forall \varphi \in L_{q_0'}(\Omega).
\]
Moreover, we have
\begin{equation}			\label{eq1742sat}
\norm{u}_{L_{q_0}(\Omega)} \le N \left(\norm{\mathbf f}_{L_{q_0}(\Omega)} + \norm{\boldsymbol g}_{L_{d q_0/(d+q_0)}(\Omega)} + \norm{h}_{L_{d q_0/(d+2 q_0)}(\Omega)} +\norm{\psi}_{L_{q_0}(\partial\Omega)}\right).
\end{equation}
We remark that \eqref{eq1742sat} is a slight generalization of \cite[Lemma 2]{EM2017}.
 
Note that for $\eta \in C^\infty(\overline \Omega)$ satisfying $\eta=0$ on $\partial \Omega$, we can take $\varphi=L\eta$, so that the identity \eqref{eq0958sat} holds.
Furthermore, we see that we can take $\eta \in \mathring W^2_{q_0'}(\Omega)$ in the identity \eqref{eq0958sat}.
We have proved the following theorem.
\begin{theorem}
Let $\Omega$ be a bounded $C^{1,1}$ domain in $\mathbb{R}^d$ with $d\ge 3$.
Assume $\mathbf A \in \mathrm{VMO}$, $\boldsymbol b \in L_{p_0}(\Omega)$, $c \in L_{p_0/2}(\Omega)$ with $p_0>d$. Additionally, assume $c \ge 0$.
Let $\mathbf f \in L_{q_0}(\Omega)$, $\boldsymbol g \in L_{d q_0/(d+q_0)}(\Omega)$,  $h \in L_{d q_0/(d+2 q_0)}(\Omega)$, and $\psi \in L_{q_0}(\partial\Omega)$, where $q_0>d/(d-2)$.
There exists a unique weak solution $u \in L_{q_0}(\Omega)$ of the problem
\[
L^*u-cu=\dv^2 \mathbf{f}+\dv \boldsymbol g+h \;\text{ in }\;\Omega,\quad u=\frac{\mathbf f \nu \cdot \nu}{\mathbf A \nu\cdot \nu}+\psi\;\text{ on }\;\partial\Omega,
\]
satisfying the estimate \eqref{eq1742sat}.
Moreover, the identity \eqref{eq0958sat} holds with any $\eta \in \mathring W^{2}_{q_0'}(\Omega)$.
\end{theorem}

\subsection{Higher integrability of weak solutions}			\label{sec2.3}
Again, we assume $d \ge 3$ for simplicity of presentation.
Let $\mathbf f \in L_{q_0}(B_{2})$, $\boldsymbol g \in L_{dq_0/(d+q_0)}(B_2)$, and $h \in L_{dq_0/(d+2q_0)}(B_2)$, where $q_0>d/(d-2)$.
Let $u \in L_1(B_{2})$ be a weak solution of
\[
L^* u -cu=\dv^2 \mathbf f+\dv \boldsymbol g+h\;\text{ in }B_2.
\]
We will show that $u \in L_{q_0}(B_1)$ if $\mathbf A \in \mathrm{DMO}$ by adopting the method from \cite{Brezis}.

By definition of a weak solution, we have
\begin{equation}				\label{eq1602sat}
\int_{B_2} u a^{ij}D_{ij}\eta + u b^i D_i \eta -uc\eta = \int_{B_2} f^{ij} D_{ij}\eta-g^i D_i \eta + h\eta, \quad \forall \eta \in C^{\infty}_c(B_2).
\end{equation}
For any $\varphi \in C^{\infty}_c(B_2)$, let $v \in \mathring W^{2}_{p}(B_2)$ be the solution of the problem
\[
a^{ij}D_{ij}v= \varphi \;\text{ in }\;B_2, \quad v=0\;\text{ on }\;\partial B_2.
\]
By \cite[Theorem 1.6]{DK17}, we find that $v \in C^2(B_2)$.
Let $v_{\epsilon}$ be the standard mollification of $v$. 
For any fixed cut-off function $\zeta \in C^\infty_c(B_{r})$ such that $\zeta=1$ in $B_{\rho}$, where $1<\rho<r<2$, we take $\eta=\zeta v_{\epsilon} \in C^\infty_c(B_2)$.
Then by \eqref{eq1602sat}, we have
\begin{multline}				\label{eq1450sat}
\int_{B_2} u a^{ij}\left(D_{ij} \zeta v_{\epsilon} + 2 D_i\zeta D_j v_{\epsilon} + \zeta D_{ij}v_{\varepsilon}\right)
+u b^i \left(D_{i} \zeta v_{\epsilon} + \zeta D_{i}v_{\epsilon}\right)+ u c\zeta v_{\epsilon}\\
=\int_{B_2} f^{ij} \left(D_{ij} \zeta v_{\epsilon} + 2 D_i\zeta D_j v_{\epsilon} + \zeta D_{ij}v_{\epsilon}\right)+ g^i \left(D_{i} \zeta v_{\epsilon} + \zeta D_{i}v_{\epsilon}\right)+h \zeta v_{\epsilon}.
\end{multline}
There is a sequence $\epsilon_n$ converging to $0$ such that
\[
v_{\epsilon_n} \to v, \quad D v_{\epsilon_n} \to Dv,\quad D^2 v_{\epsilon_n} \to D^2 v \quad \text{ a.e. in $B_r$.}
\]
Also, note that for any $\epsilon \in (0, \frac{2-r}{2})$ we have
\[
\norm{v_{\epsilon}}_{W^{2}_\infty(B_r)} \le \norm{v}_{W^{2}_\infty(B_{(r+2)/2})}<\infty.
\]
Hence, by using the dominated convergence theorem, we find that \eqref{eq1450sat} is valid with $v^{(\epsilon)}$ replaced by $v$; i.e.,
\begin{multline}				\label{eq2056sat}
\int_{B_2} u a^{ij}(D_{ij} \zeta v + 2 D_i\zeta D_j v+ \zeta D_{ij}v)
+ u b^i (D_{i} \zeta v+ \zeta D_{i}v)+  u c\zeta v\\
=\int_{B_2}  f^{ij} (D_{ij} \zeta v + 2 D_i\zeta D_j v+ \zeta D_{ij}v)
+ g^i (D_{i} \zeta v+ \zeta D_{i}v)+  h\zeta v.
\end{multline}

On the other hand, by multiplying $\zeta u$ to \eqref{eq1602sat} and noting that $D^2 v$, $\varphi \in L_\infty(B_r)$ and $\zeta u \in L_1(B_2)$,  we have
\begin{equation}			\label{eq1627sat}
\int_{B_2} a^{ij} D_{ij}v \zeta u  = \int_{B_2} \varphi \zeta u.
\end{equation}
By combing \eqref{eq1627sat} and \eqref{eq2056sat}, we obtain
\begin{multline}			\label{eq0959sun}
\int_{B_2} \zeta u \varphi=\int_{B_2}  f^{ij}(D_{ij} \zeta v + 2 D_i\zeta D_j v+ \zeta D_{ij}v)
+ g^i (D_{i} \zeta v+ \zeta D_{i}v)+  h\zeta v\\
-\int_{B_2} u a^{ij}(D_{ij}\zeta v+ 2 D_i\zeta D_jv)+ ub^i \left(D_i\zeta v +\zeta D_iv\right)+uc\zeta v.
\end{multline}
Note that we can choose $\zeta$ such that
\[
\norm{\zeta}_\infty \le1,\quad \norm{D \zeta}_\infty  \le 2/(r-\rho),\quad \norm{D^2 \zeta}_\infty \le 4/(r-\rho)^2.
\]

Take $r=r_0$ and $\rho=r_1$, where $1<r_0<r_1<2$.
It follows from \eqref{eq0959sun} that
\begin{multline}			\label{eq1650sat}
\Abs{\int_{B_2} \zeta u \varphi\,} \lesssim 
\norm{\mathbf f}_{L_{q_0}(B_2)} \norm{v}_{W^2_{q_0'}(B_2)}+\norm{\boldsymbol g}_{L_{dq_0/(d+q_0)}(B_2)} \norm{v}_{W^2_{q_0'}(B_2)} +\norm{h}_{L_{dq_0/(d+2q_0)}(B_2)} \norm{v}_{W^2_{q_0'}(B_2)}\\
+\norm{u}_{L_1(B_2)}\norm{v}_{W^1_\infty(B_2)}+\norm{u \boldsymbol b}_{L_1(B_2)} \norm{v}_{W^1_\infty(B_2)}+\norm{u c}_{L_1(B_2)} \norm{v}_{L_\infty(B_2)}.
\end{multline}
Here, we use the notation $A \lesssim B$ to mean that $A \le N(r_1-r_0)^{-2} B$.

By the $L_p$ theory, for any $p \in (1,\infty)$, we have
\begin{equation}			\label{eq0941tue}
\norm{v}_{W^{2}_{p}(B_2)} \le N \norm{\varphi}_{L_{p}(B_2)}.
\end{equation}
Then it follows from the Sobolev inequalty that for any $p_1' \in (d,\infty)$, we have
\begin{equation}			\label{eq0944tue}
\norm{v}_{W^1_\infty(B_2)} \le N \norm{v}_{W^{2}_{p_1'}(B_2)} \le  N \norm{\varphi}_{L_{p_1'}(B_2)}.
\end{equation}
Note that $q_0'<p_1'$.
Since $\varphi \in C^\infty_c(B_2)$ is arbitrary, by utilizing \eqref{eq0944tue}, \eqref{eq1650sat} and applying the converse of H\"older's inequality, we derive
\begin{multline*}
\norm{u}_{L_{p_1}(B_{r_1})}\le \norm{\zeta u}_{L_{p_1}(B_2)} \lesssim \norm{\mathbf f}_{L_{q_0}(B_2)}+\norm{\boldsymbol g}_{L_{dq_0/(d+q_0)}(B_2)}+\norm{h}_{L_{dq_0/(d+2q_0)}(B_2)}\\
+\norm{u}_{L_1(B_2)}+ \norm{u\boldsymbol b}_{L_1(B_2)}+\norm{cu}_{L_1(B_2)},\quad \forall p_1 \in (1, d/(d-1)).
\end{multline*}

We can iterate this process.
Let $k$ be the largest integer such that
\[
1/d+k(1/d-1/p_0)\le 1/q_0'.
\]
We take $p_1'>d$ sufficiently close to $d$ such that
\[
1/p_1'+k(1/d-1/p_0)<1/q_0'\quad\text{and}\quad 1/p_1'+(k+1)(1/d-1/p_0)>1/q_0'.
\]
For $j=1,\ldots,k$, we set
\[
1/p_{j}'=1/p_1'+(j-1)(1/d-1/p_0),\quad 1/s_{j}=1/p_{j}'-1/p_0, \quad 1/s_{j}^*=1/s_{j}-1/d.
\]
Taking $r_{j+1}$ and $r_j$ in the place of $\rho$ and $r$, and going back to \eqref{eq0959sun}, we see that
\begin{multline}			\label{eq0957tue}	
\Abs{\int_{B_2} \zeta u \varphi\,} \lesssim
\left(\norm{\mathbf f}_{L_{q_0}(B_2)} +\norm{\boldsymbol g}_{L_{dq_0/(d+q_0)}(B_2)} +\norm{h}_{L_{dq_0/(d+2q_0)}(B_2)}\right) \norm{v}_{W^2_{q_0'}(B_2)}\\
+\norm{u}_{L_{p_{j}}(B_{r_j})} \left(\norm{v}_{W^1_{s_j}(B_2)}+\norm{\boldsymbol b}_{L_{p_0}(B_2)} \norm{v}_{W^1_{s_j}(B_2)} +\norm{c}_{L_{p_0/2}(B_2)} \norm{v}_{L_{s_j^*}(B_2)} \right),
\end{multline}
where we used 
\[
1/p_j+1/p_0+1/s_j=1,\quad1/p_j+2/p_0+1/s_j^*<1.
\]
Note that for $j=1,\ldots, k$, we have
\[
1/p_{j+1}'=1/s_j+1/d\quad\text{and}\quad 1/p_{j+1}'<1/q_0'.
\]
Hence, it follows from \eqref{eq0941tue} and the Sobolev inequality that
\[
\norm{v}_{L_{s_j^*}(B_2)}+ \norm{v}_{W^1_{s_j}(B_2)} \le N \norm{\varphi}_{L_{p_{j+1}'}(B_2)},\quad
\norm{v}_{W^2_{q_0'}(B_2)} \le  N \norm{\varphi}_{L_{p_{j+1}'}(B_2)}.
\]
Therefore, we derive from \eqref{eq0957tue} and the converse of H\"older's inequality that
\begin{multline*}
\norm{u}_{L_{p_{j+1}}(B_{r_{j+1}})}\le \norm{\zeta u}_{L_{p_{j+1}}(B_2)} \lesssim \norm{\mathbf f}_{L_{q_0}(B_2)}+\norm{\boldsymbol g}_{L_{dq_0/(d+q_0)}(B_2)}+\norm{h}_{L_{dq_0/(d+2q_0)}(B_2)}\\
+\norm{u}_{L_{p_j}(B_{r_j})}\left(1+ \norm{\boldsymbol b}_{L_{p_0}(B_2)}+\norm{c}_{L_{p_0/2}(B_2)}\right),\quad j=1,\ldots,k.
\end{multline*}
Here, we use the notation $A \lesssim B$ to mean that $A \le N(r_{j+1}-r_j)^{-2} B$.

In the case when $j=k+1$, we set
\[
1/p_{k+2}'=1/q_0',\quad1/s_{k+1}=1/q_0'-1/d,\quad 1/s_{k+1}^*=1/s_{k+1}-1/d.
\]
Taking $r=r_{k+1}$ and $\rho=r_{k+2}$, similar to \eqref{eq0957tue}, we have
\begin{multline*}	
\Abs{\int_{B_2} \zeta u \varphi\,} \lesssim
\left(\norm{\mathbf f}_{L_{q_0}(B_2)} +\norm{\boldsymbol g}_{L_{dq_0/(d+q_0)}(B_2)} +\norm{h}_{L_{dq_0/(d+2q_0)}(B_2)}\right) \norm{v}_{W^2_{q_0'}(B_2)}\\
+\norm{u}_{L_{p_{k+1}}(B_{r_{k+1}})} \left(\norm{v}_{W^1_{s_{k+1}}(B_2)}+\norm{\boldsymbol b}_{L_{p_0}(B_2)} \norm{v}_{W^1_{s_{k+1}}(B_2)} +\norm{c}_{L_{p_0/2}(B_2)} \norm{v}_{L_{s_{k+1}^*}(B_2)} \right),
\end{multline*}
where we used 
\[
1/p_{k+1}+1/p_0+1/s_{k+1}<1,\quad 1/p_{k+1}+2/p_0+1/s_{k+1}^*<1.
\]
Noting $\norm{v}_{W^1_{s_{k+1}}(B_2)} \le N \norm{\varphi}_{L_{q_{0}'}(B_2)}$, and invoking the converse of H\"older's inequality, we obtain
\begin{multline*}
\norm{u}_{L_{q_{0}}(B_{r_{k+2}})}\le \norm{\zeta u}_{L_{q_{0}}(B_2)} \lesssim \norm{\mathbf f}_{L_{q_0}(B_2)}+\norm{\boldsymbol g}_{L_{dq_0/(d+q_0)}(B_2)}+\norm{h}_{L_{dq_0/(d+2q_0)}(B_2)}\\
+\norm{u}_{L_{p_{k+1}}(B_{r_{k+1}})}\left(1+ \norm{\boldsymbol b}_{L_{p_0}(B_2)}+\norm{c}_{L_{p_0/2}(B_2)}\right).
\end{multline*}

After applying an affine transformation if necessary, we have established the following theorem.
\begin{theorem}	\label{thm2.14}
Let $\mathbf A \in \mathrm{DMO}$, $\boldsymbol b \in L_{p_0}(B_{r})$, $c \in L_{p_0/2}(B_{r})$ with $p_0>d$.
Let $\mathbf f \in L_{q_0}(B_{r})$, $\boldsymbol g \in L_{dq_0/(d+q_0)}(B_{r})$, and $h \in L_{dq_0/(d+2q_0)}(B_{r})$, where $q_0>d/(d-2)$.
Suppose $u \in L_{1}(B_{r})$ is a weak solution of
\[
L^*u-cu=\dv^2 \mathbf{f}+\dv \boldsymbol g+h\;\text{ in }\;B_{r}.
\]
Then we have $u \in L_{q_0}(B_{\rho})$ for any $\rho<r$.
\end{theorem}

\subsection{Continuity}			
We begin with stating the $C^0$ estimates established in \cite{DK17}, where we consider the case when $\boldsymbol{b}=0$ and $c=0$.

\begin{theorem}
Suppose that  $\mathbf A  \in \mathrm{DMO}$.
Let $u \in L_1(B_2)$ is a weak solution of
\[
\dv (\mathbf A u)= \dv^2 \mathbf f \;\mbox{ in } \; B_2,
\]
where $\mathbf f=(f^{kl})$ is of Dini mean oscillation in $B_2$.
Then, we have $u\in C^0(\overline B_1)$.
\end{theorem}

We provide some comments on the proof.
In \cite[Theorem 1.10]{DK17}, it assumes that $u \in L_2(B_4)$ instead of $u \in L_1(B_2)$ in obtaining the conclusion that $u$ is continuous in $B_1$.
Firstly, it is worth noting that Theorem \ref{thm2.14} establishes that  $u \in L_2(B_{3/2})$.
Secondly, the standard covering argument renders the specific sizes of radii irrelevant.

Also, in the proof of \cite[Theorem 1.10]{DK17}, it is implicitly assumed that $u$ is locally bounded.
We explain why this assumption is unnecessary.
Let $\mathbf A_n$ and $\mathbf f_n$ be smooth functions, uniformly converging to $\mathbf A$ and $\mathbf f$, respectively.
Let $u_n$ be the solution of $\dv^2(\mathbf A_n u)=\dv^2 \mathbf f_n$ in $B_2$ with boundary condition $u_n=\mathbf f_n \nu \cdot \nu/ \mathbf A_n \nu\cdot \nu$, and let $u$ be the solution of $\dv^2(\mathbf A u)=\dv^2 \mathbf f$ in $B_2$ with boundary condition $u=\mathbf f \nu \cdot \nu/ \mathbf A \nu\cdot \nu$.
Note that we have
\[
\dv^2(\mathbf A_n(u-u_n))=\dv^2 (\mathbf f-\mathbf f_n)+\dv^2((\mathbf A_n-\mathbf A)u)\;\text{ in }\;B_2
\]
and $\mathbf A_n(u-u_n)\nu\cdot \nu =(\mathbf f-\mathbf f_n)\nu\cdot \nu+ ((\mathbf A_n-\mathbf A)u)\nu\cdot \nu$ on $\partial B_2$.
Then, by the $L_p$ estimates, we have
\[
\norm{u-u_n}_{L_p(B_2)} \le N \norm{\mathbf f- \mathbf f_n}_{L_p(B_2)}+N \norm{(\mathbf A_n- \mathbf A)u}_{L_p(B_2)},
\]
where $N$ depends on $d$, $\delta$, and the modulus of continuity of $\mathbf A_n$, which are uniformly bounded in $n$. In particular, $N$ is independent of $n$.
We have $\mathbf f_n \to \mathbf f$ in $L_p(B_2)$.
Also, the dominated convergence theorem implies that 
\[
\norm{(\mathbf A_n- \mathbf A)u}_{L_p(B_2)} \to 0
\]
because $\mathbf A_n \to \mathbf A$ a.e. and $\abs{(\mathbf A_n- \mathbf A)u} \le 2\delta^{-1} \abs{u} \in L_p(B_2)$.
This implies $u_n \to u$ in $L_p(B_2)$.
By passing to a subsequence, we have $u_n \to u$ a.e. in $B_2$.
In particular, we find that $u$ is locally bounded since it follows from the proof of \cite[Theorem 1.10]{DK17} that $u_n$ satisfies the following uniform estimate:
\[
\norm{u_n}_{L_\infty(B_{3/2})} \lesssim  \norm{u_n}_{L_1(B_2)} +  \int_0^{1}  \frac{\tilde\omega_{\mathbf f_n}(t)}t \,dt \lesssim \norm{u}_{L_1(B_2)} + \int_0^{1} \frac{\tilde\omega_{\mathbf f}(t)}t \,dt.
\]

Next, we will discuss the global $C^0$ estimate.
Let $\Omega$ be a bounded $C^{1,1}$ domain in $\mathbb{R}^d$.
Again, for simplicity, we only consider the case when $d\ge 3$.
Let $\mathbf f \in \mathrm{DMO}$, $\boldsymbol g \in L_{q_0}(\Omega)$, and $h \in L_{q_0/2}(B_2)$, where $q_0>d$.
Let $\psi$ be a continuous function on $\partial \Omega$.
Let $u\in L_1(\Omega)$ be a weak solution of
\[
L^*u-cu=\dv^2 \mathbf{f}+\dv \boldsymbol g+h \;\text{ in }\;\Omega,\quad u=\frac{\mathbf f \nu \cdot \nu}{\mathbf A \nu\cdot \nu}+\psi\;\text{ on }\;\partial\Omega.
\]
We will demonstrate that $u \in C(\overline{\Omega})$.
This will slightly enhance \cite[Theorem 1.8]{DEK18} by requiring only that $u \in L_1(\Omega)$ and incorporating $\boldsymbol{g}$ as an inhomogeneous term.
However, there is essentially no difference in the proof, except for rectifying a flaw in the proof of \cite[Proposition 2.11]{DEK18}.
In that proof, it is assumed that $u \in L_2(\Omega)$ and $c \in L_{p_0/2}(\Omega)$ with $p_0 > d$.
However, these assumptions do not exclude the case where $cu \not\in L_q(\Omega)$ for any $q > 1$.
This constitutes a flaw, which we remedy here.

First, let $\mathbf f \in L_{q_0}(\Omega)$, $\boldsymbol g \in L_{dq_0/(d+q_0)}(\Omega)$, and $h \in L_{dq_0/(d+2q_0)}(\Omega)$, where $q_0>d/(d-2)$, and consider the following problem:
\[
L^*u-cu=\dv^2 \mathbf{f}+\dv \boldsymbol g+h \;\text{ in }\;\Omega,\quad u=\frac{\mathbf f \nu \cdot \nu}{\mathbf A \nu\cdot \nu}\;\text{ on }\;\partial\Omega.
\]
We will show that if $u \in L_1(\Omega)$ is a weak solution of the problem, then $u \in L_{q_0}(\Omega)$.

Fix a point $x_0 \in \partial \Omega$, and let $\boldsymbol \Phi=(\phi^1,\ldots, \phi^d)$ be a $C^{1,1}$ diffeomorphism in $\mathbb{R}^d$ that maps $\Omega$ to $\tilde \Omega:=\boldsymbol \Phi(\Omega)$ such that $\tilde \Omega \cap B_r(y_0)=B_r^+(y_0)$, where $y_0=\boldsymbol\Phi(x_0)$ and
\[
B_r^+(y_0)= B_r(y_0) \cap \mathbb{R}_+^d=\set{y \in \mathbb{R}^d: \abs{y-y_0}<r,\; y^d>0}.
\]
Let $\boldsymbol \Psi=(\psi^1,\ldots, \psi^d)$ be the inverse of $\boldsymbol \Phi$.
Recall the change of variables formula:
\[
\int_{\Omega} f(x) \,dx= \int_{\tilde \Omega} f(\boldsymbol \Psi (y)) \abs{\det D\boldsymbol \Psi(y)}\,dy.
\]

We aim to express the identity \eqref{eq0958sat} in new coordinates using the change of variables.
Observe that in the identity \eqref{eq0958sat}, we can take $\eta$ to be any $C^{1,1}$ function satisfying $\eta=0$ on $\partial\Omega$.
To see this, observe that for $\eta \in C^{1,1}$, there exists a sequence of smooth functions $\eta^{(n)} \in C^\infty(\overline \Omega)$ vanishing on $\partial\Omega$ such that 
\[
\eta^{(n)} \to \eta, \;\; D \eta^{(n)} \to D\eta,\;\; D^2 \eta^{(n)} \to D^2 \eta \;\;\text{ a.e.}\quad\text{and}\quad \norm{\eta^{(n)}}_{C^{1,1}(\Omega)} \le \norm{\eta}_{C^{1,1}(\Omega)}.
\]
The same argument in Section \ref{sec2.3} shows identity \eqref{eq0958sat} is valid with $\eta \in C^{1,1}(\overline\Omega)$.

Let $\tilde \eta(y)= \eta(\boldsymbol\Psi(y))$ and $\tilde u(y)= u(\boldsymbol\Psi(y))$.
By writing $x=\boldsymbol\Psi(y)$ and $y=\boldsymbol\Phi(x)$, we have 
\[
\eta(x)=\tilde \eta(y)=\tilde \eta(\boldsymbol\Phi(x)),\quad u(x)=\tilde u(y)=\tilde u(\boldsymbol\Phi(x)).
\]
Note that $\tilde \eta$ is a $C^{1,1}$ function.
By the chain rule, we have
\begin{align*}
D_{i} \eta(x)&=D_{k} \tilde \eta(\boldsymbol\Phi(x)) D_i \phi^k(x),\\
D_{ij} \eta(x)&= D_{kl} \tilde \eta(\boldsymbol\Phi(x)) D_i \phi^k(x) D_j \phi^l(x)+D_k \tilde \eta(\boldsymbol \Phi(x))D_{ij}\phi^k(x),
\end{align*}
Therefore, the integrals in \eqref{eq0958sat} becomes
\begin{equation}			\label{eq1619wed}
\int_{\tilde\Omega} \tilde u \tilde a^{kl} D_{kl}\tilde \eta + \tilde u \tilde b^k D_k \tilde \eta -\tilde u \tilde c \tilde \eta\,dy= \int_{\tilde\Omega} \tilde f^{kl} D_{kl} \tilde \eta -\tilde g^k D_k \tilde \eta + \tilde h \tilde \eta\,dy,
\end{equation}
where we set
\begin{align*}
\tilde a^{kl}(y)&= a^{ij}(\boldsymbol\Psi(y)) D_i\phi^k(\boldsymbol\Psi(y)) D_j\phi^l(\boldsymbol\Psi(y))\,\abs{\det D\boldsymbol\Psi(y)},\\
\tilde b^k(y)&= \left\{b^{i}(\boldsymbol\Psi(y)) D_i \phi^k(\boldsymbol\Psi(y))+ a^{ij}(\boldsymbol\Psi(y))D_{ij}\phi^k(\boldsymbol\Psi(y))\right\} \abs{\det D\boldsymbol\Psi(y)},\\
\tilde c(y)&= c(\boldsymbol\Psi(y))\,\abs{\det D\boldsymbol\Psi(y)},\\
\tilde f^{kl}(y) &= f^{ij}(\boldsymbol\Psi(y)) D_i\phi^k(\boldsymbol\Psi(y)) D_j\phi^l(\boldsymbol\Psi(y)) \,\abs{\det D\boldsymbol\Psi(y)},\\
\tilde g^k(y) &= \left\{g^i(\boldsymbol\Psi(y)) D_i \phi^k(\boldsymbol\Psi(y)) +f^{ij}(\boldsymbol\Psi(y))D_{ij}\phi^k(\boldsymbol\Psi(y))\right\} \,\abs{\det D\boldsymbol\Psi(y)},\\
\tilde h(y)&=h(\boldsymbol\Psi(y))\, \abs{\det D\boldsymbol\Psi(y)}.
\end{align*}

We can take $\eta=\zeta \circ \boldsymbol \Phi$, where $\zeta$ is any smooth function on $\tilde\Omega$ vanishing on $\partial{\tilde \Omega}$, so that $\tilde \eta=\zeta$.
Hence, it follows from \eqref{eq1619wed} that $\tilde u \in L_1(\tilde \Omega)$ is a weak solution of the problem
\[
\dv^2(\tilde{\mathbf A}\tilde u)-\dv(\tilde{\boldsymbol b}u)-\tilde c \tilde u=\dv^2 \tilde{\mathbf f}+\dv \tilde{\boldsymbol g}+\tilde{h} \;\text{ in }\;\tilde\Omega,\quad \tilde u=\frac{\tilde{\mathbf f} \nu \cdot \nu}{\tilde{\mathbf A} \nu\cdot \nu}\;\text{ on }\;\partial\tilde\Omega.
\]
Note that transformed coefficients and data belong to the same classes as the original ones.
Therefore, using dilation if necessary, we may assume that $\Omega$ is such that $\Omega\cap B_2(x_0)=B_2^+(x_0)$.
Then, for any $\eta \in C^{\infty}_c(B_2)$ vanishing on $B_2 \cap \set{x_d=0}$, we have
\begin{equation}				\label{eq0926wed}
\int_{B_2^+} u a^{ij}D_{ij}\eta + u b^i D_i \eta -uc\eta = \int_{B_2^+} f^{ij} D_{ij}\eta-g^i D_i \eta + h\eta.
\end{equation}
Consider a smooth domain $U$ that is a subset of $\mathbb{R}^d_+ \cap \Omega$ and encloses $B_2^+$.
For any $\varphi \in C^{\infty}_c(B_2^+)$, let $v \in \mathring W^{2}_{p}(U)$ be the solution of the problem
\[
a^{ij}D_{ij}v= \varphi \;\text{ in }\;U, \quad v=0\;\text{ on }\;\partial U.
\]
The standard $L_p$ theory shows that
\[
\norm{v}_{W^2_p(B_2^+)}  \le \norm{v}_{W^2_p(U)}  \le N \norm{\varphi}_{L_p(U)}=N \norm{\varphi}_{L_p(B_2^+)},\quad \forall p \in (1,\infty).
\]
Also, by \cite[Proposition 2.15]{DEK18}, we find that $v \in C^2(\overline{B_r^+})$ for any $r \in (0,2)$.
We extend $v$ on $B_2$ by odd extension, and let $v^{(\epsilon)}$ be the standard mollification of $v$. 
Notice that $v^{(\epsilon)}=0$ on $B_2 \cap \set{x_d=0}$. 
For a function $\zeta \in C^\infty_c(B_{r})$ such that $\zeta=1$ in $B_{\rho}$, where $1<\rho<r<2$ are to be fixed, we take $\eta=\zeta v^{(\epsilon)} \in C^\infty_c(B_2)$.
Then by \eqref{eq0926wed}, we derive \eqref{eq1450sat} with $B_2$ replaced by $B_2^+$.
Since for any $\epsilon \in (0, \frac{2-r}{2})$ we have
\[
\norm{v^{(\epsilon)}}_{W^{2}_\infty(B_r^+)} \le \norm{v}_{W^{2}_\infty(B_{(r+2)/2}^+)}<\infty,
\]
the dominated convergence theorem implies that \eqref{eq2056sat} is valid with $B_2^+$ replaced by $B_2$.
We note that the Sobolev embedding holds for $B_2^+$ in place of $B_2$.
Therefore, by repeating the same argument as in Section \ref{sec2.3}, we conclude that $u \in L_{q_0}(B_1^{+})$.
By using the partition of unity, and using the interior results from Section \ref{sec2.3}, we conclude that $u \in L_{q_0}(\Omega)$.

\begin{theorem}	\label{thm2.16}
Let $\Omega$ be a bounded $C^{1,1}$ domain in $\mathbb{R}^d$ with $d\ge 3$.
Let $\mathbf A \in \mathrm{DMO}$, $\boldsymbol b \in L_{p_0}(\Omega)$, $c \in L_{p_0/2}(\Omega)$ with $p_0>d$.
Let $\mathbf f \in L_{q_0}(\Omega)$, $\boldsymbol g \in L_{dq_0/(d+q_0)}(\Omega)$, and $h \in L_{dq_0/(d+2q_0)}(\Omega)$, where $q_0>d/(d-2)$.
Suppose $u \in L_{1}(\Omega)$ is a weak solution of
\[
L^*u-cu=\dv^2 \mathbf{f}+\dv \boldsymbol g+h \;\text{ in }\;\Omega,\quad u=\frac{\mathbf f \nu \cdot \nu}{\mathbf A \nu\cdot \nu}\;\text{ on }\;\partial\Omega.
\]
Then we have $u \in L_{q_0}(\Omega)$.
\end{theorem}

Next, we rearrange terms and write the equation as follows.
\[
\dv^2(\mathbf A u)=\dv(\boldsymbol b u)+cu+\dv^2 \mathbf f + \dv \boldsymbol g +h\;\text{ in }\;\Omega.
\]
We will express the RHS as $\dv^2 (\mathbf f+w \mathbf I)$, where $w$ is the solution of the problem
\[
\Delta w= \dv(\boldsymbol bu+ \boldsymbol g)+cu+h\;\text{ in }\;\Omega,\quad w=0 \;\text{ on }\;\partial \Omega.
\]
By Theorem \ref{thm2.16}, we deduce that $u \in L_{q}(B)$ for all $q \in (1,\infty)$.
Since $\boldsymbol b \in L_{p_0}(\Omega)$ and $c \in L_{p_0/2}(\Omega)$ with $p_0>d$, we find that $\boldsymbol b u +\boldsymbol g \in L_p(\Omega)$ and $cu+h \in L_{p/2}(\Omega)$ for some $p>d$.
By the $L_p$ theory and the Sobolev imbedding, we deduce that $w \in C^{\alpha}$ for $\alpha>0$.
In particular, we see that $w\mathbf I \in \mathrm{DMO}$.
Moreover, we note that $(w\mathbf I) \nu\cdot \nu =w=0$ on $\partial \Omega$.
We have shown that $u$ satisfies the following:
\[
\dv^2(\mathbf A u)=\dv^2(\mathbf f + w\mathbf I)\;\text{ in }\;\Omega,\quad u=\frac{(\mathbf f + w \mathbf I) \nu \cdot \nu}{\mathbf A \nu\cdot \nu} + \psi\;\text{ on }\;\partial\Omega.
\]
Here, $\mathbf A$ and $\mathbf f + w\mathbf I$ are functions of Dini mean oscillation.
Recall that for any $x_0 \in \Omega$, by flattening and dilation, we may assume that $\Omega\cap B_2(x_0)=B_2^+(x_0)$.
By combining \cite[Lemma 2.22]{DEK18}, \cite[Proposition 2.23]{DEK18}, and \cite[Theorem 1.10]{DK17}, we find that $u \in C(\overline \Omega)$.
We have proved the following theorem.

\begin{theorem}			\label{thm3.17}
Let $\Omega$ be a bounded $C^{1,1}$ domain in $\mathbb{R}^d$ with $d\ge 3$.
Assume that $\mathbf A \in \mathrm{DMO}$, $\boldsymbol b \in L_{p_0}(\Omega)$, $c \in L_{p_0/2}(\Omega)$ with $p_0>d$.
Let $\mathbf f \in \mathrm{DMO}(\Omega)$, $\boldsymbol g \in L_{q_0}(\Omega)$, and $h \in L_{q_0/2}(\Omega)$, where $q_0>d$.
Let $\psi \in C(\partial\Omega)$.
Suppose $u \in L_{1}(\Omega)$ is a weak solution of
\[
L^*u-cu=\dv^2 \mathbf{f}+\dv \boldsymbol g+h \;\text{ in }\;\Omega,\quad u=\frac{\mathbf f \nu \cdot \nu}{\mathbf A \nu\cdot \nu}+\psi\;\text{ on }\;\partial\Omega.
\]
Then we have $u \in C^0(\overline \Omega)$.
\end{theorem}

\subsection{Harnack inequality}
The following theorem is a restatement of \cite[Lemma 4.2]{DEK18}, taking account of Theorem \ref{thm2.14}.
\begin{theorem}
Assume the coefficients $\mathbf{A} \in \mathrm{DMO}$.
Let $u\in L_1(B_2)$ be a nonnegative solution to $\dv(\mathbf A u)=0$ in $B_2$.
Then we have
\[
\sup_{B_1}\,u \le N \inf_{B_1}\,u,
\]
where $N$ is a positive constant depending only on $d$, $\delta$, and $\omega_{\mathbf{A}}$.
\end{theorem}

In a recent article \cite{BRS23}, the previous Harnack inequality is extended to nonnegative solutions of $\dv^2(\mathbf Au) -\dv(\boldsymbol b u)-cu=0$ with $\boldsymbol b$, $c \in L_{p_0}$ for $p_0>d$.
We can further relax the assumption on $c$ by requiring $c \in L_{p}$ with $p>d/2$.
More precisely, the following theorem is valid.
Refer to Section \ref{sec4.2} for the parabolic counterpart.

\begin{theorem}
Assume the coefficients $\mathbf{A} \in \mathrm{DMO}$, $\boldsymbol b \in L_{p_0}(B_2)$, and $c \in L_{p_0/2}(B_2)$ for some $p_0>d$.
Let $u\in L_1(B_2)$ be a nonnegative solution to
\[
\dv(\mathbf A u)-\dv(\boldsymbol b u)-cu=0\;\text{ in }\;B_2.
\]
Then we have
\[
\sup_{B_1}\,u \le N \inf_{B_1}\,u,
\]
where $N$ is a positive constant depending only on $d$, $\delta$, $\omega_{\mathbf{A}}$, $\norm{\boldsymbol b}_{L_{p_0}(B_2)}$, $\norm{c}_{L_{p_0/2}(B_2)}$, and $p_0$.
\end{theorem}

\section{Elliptic equations with coefficients in Morrey classes}			\label{sec3}
We assume that the lower-order coefficients $\boldsymbol b$ and $c$ belong to the Morrey spaces such that Theorem 2.8 in \cite{Krylov2023e} can be applied.
In particular, we require that $\boldsymbol b$ and $c$ belong to  $E_{q_b, 1}$ and $E_{q_c, 2}$, respectively, where $q_b \in (d/2, d]$ and $q_c \in (1,d/2]$, and that $c \ge 0$.
See \cite{Krylov2023e} for details.
Additionally, we assume that
\begin{equation}		\label{omega_coef}
\omega_{\rm coef}(r):=  \omega_{\mathbf A}(r)+  r\sup_{x \in \Omega} \fint_{\Omega \cap B_r(x)}\abs{\boldsymbol b}+ r^2\sup_{x \in \Omega} \fint_{\Omega \cap B_r(x)}\abs{c}
\end{equation}
satisfies the Dini condition; i.e.
\[
\int_0^{1} \frac{\omega_{\rm coef}(t)}{t}\,dt<+\infty.
\]

Also, we assume that
$\mathbf f \in \mathrm{DMO}(\Omega)$, $\boldsymbol g \in L_{q_0}(\Omega)$, and $h \in L_{q_0/2}(\Omega)$, where $q_0>d$, so that 
\begin{equation}			\label{omega_dat}
\omega_{\rm dat}(r):=  \omega_{\mathbf f}(r)+ r\sup_{x \in \Omega} \fint_{\Omega \cap B_r(x)}\abs{\boldsymbol g} + r^2 \sup_{x \in \Omega} \fint_{B_r(x) \cap \Omega}\abs{h}
\end{equation}
satisfies the Dini condition.

Let $u \in L_1(\Omega)$ be a weak solution of
\[
\left\{
\begin{aligned}
\dv^2(\mathbf A u)-\dv(\boldsymbol{b} u)-cu&=\dv^2 \mathbf f + \dv \boldsymbol{g}+h\;\text{ in }\Omega,\\
(\bar{\mathbf A} \nu \cdot \nu)u &=\mathbf{f}\nu \cdot \nu\;\text{ on }\;\partial \Omega.
\end{aligned}
\right.
\]
We will first derive modulus of continuous estimates for  $u$ under qualitative assumption that $\boldsymbol{b}$ and $c$ are bounded functions.
Then, it will follow from Theorem \ref{thm3.17} that $u$ is a bounded function.

By modifying the proof of \cite[Theorem 1.10]{DK17}, we will first establish the interior estimates, under the assumption that $B_4 \Subset \Omega$, which is

For $x_0 \in B_3$ and $0<r<\frac13$, let
\[
\bar{\mathbf A}=\fint_{B_r(x_0)} \mathbf A\quad \text{and}\quad
\bar{\mathbf f}=\fint_{B_r(x_0)} \mathbf f.
\]
We write the equation as
\[
\dv^2(\bar{\mathbf A} u)=\dv^2((\bar{\mathbf A}-\mathbf A) u)+\dv(\boldsymbol{b} u)+cu+\dv^2 (\mathbf f -\bar{\mathbf f}) + \dv \boldsymbol{g}+h.
\]
We decompose $u=v+w$, where $w$ is a unique solution of the problem
\[
\left\{
\begin{aligned}
\dv^2(\bar{\mathbf A} w) &= \dv^2 ((\bar{\mathbf A}-\mathbf A) u +\mathbf f -\bar{\mathbf f}) + \dv(\boldsymbol{b} u+\boldsymbol{g})+cu+h \;\text{ in }\;B_r(x_0),\\
(\bar{\mathbf A} \nu \cdot \nu)w &=(\mathbf{f}-\bar{\mathbf f}-(\mathbf{A}-\bar{\mathbf A})u) \nu \cdot \nu\;\text{ on }\;\partial B_r(x_0),
\end{aligned}
\right.
\]
\begin{lemma}				\label{lem-weak11-adj}
Let $B=B_1(0)$.
Let $\mathbf A_0$ be a constant symmetric matrix satisfying the condition \eqref{ellipticity-nd}.
Let $p_0>1$.
For  $\mathbf f \in L_{p_0}(B)$, $\boldsymbol{g} \in L_{p_0}(B)$, and $h \in L_{p_0}(B)$, let $u \in L_{p_0}(B)$ be the weak solution to the problem
\[
\left\{
\begin{aligned}
\dv^2(\mathbf A_0 u)&= \dv^2 \mathbf f+ \dv \boldsymbol g + h\;\mbox{ in }\; B,\\
(\mathbf A_0 \nu \cdot \nu) u &= \mathbf f \nu\cdot \nu \;\mbox{ on } \; \partial B.
\end{aligned}
\right.
\]
Then for any $t>0$, we have
\[
t \Abs{\set{x \in B : \abs{u(x)} > t}}  \le C \left(\int_{B} \abs{\mathbf f}+ \abs{\boldsymbol g}+ \abs{h} \right),
\]
where $C=C(d, \beta, p_0)$.
\end{lemma}
\begin{proof}
Refer to \cite[Lemma 3.3]{CDKK24} for the proof when $p_0=2$.
The same proof works for any $p_0>1$.
\end{proof}

By Lemma \ref{lem-weak11-adj} with scaling, we have
\begin{multline*}
t\Abs{\set{x\in B_r(x_0): \abs{w(x)} > t}} \lesssim \norm{u}_{L_\infty(B_r(x_0))}\left( \int_{B_r(x_0)} \abs{\mathbf A-\bar{\mathbf A}} + r\int_{B_r(x_0)} \abs{\boldsymbol b} +r^2\int_{B_r(x_0)} \abs{c} \right)\\
+ \int_{B_r(x_0)} \abs{\mathbf f -\bar{\mathbf f}} 
+ r\int_{B_r(x_0)} \abs{\boldsymbol g}+ r^2\int_{B_r(x_0)} \abs{h}.
\end{multline*}
Therefore, by using \eqref{omega_coef} and \eqref{omega_dat}, we have
\begin{equation}				\label{eq1622thu}
\left(\fint_{B_r(x_0)} \abs{w}^{\frac12} \right)^{2} \le N \omega_{\rm coef}(r) \,\norm{u}_{L_\infty(B_r(x_0))} + N \omega_{\rm dat}(r).
\end{equation}
We remark that the exponent $\frac{1}{2}$ in the previous inequality is arbitrary and it can be replaced by any $p \in (0,1)$.
Note that $v$ is a weak solution of 
\[
\dv^2(\bar{\mathbf A} v) = \dv^2 \bar{\mathbf f}=0  \;\text{ in }B_r=B_r(x_0),
\]
and so is $v-c$ for any constant $c \in \bR$.
Since the coefficients $\bar{\mathbf A}$ are constants, we observe that $v-c$ is a classical solution.
Therefore, we have
\[
\norm{D v}_{L_\infty(B_{r/2})} \le N_0 r^{-1} \left(\fint_{B_r} \abs{v-c}^{\frac12} \right)^{2}
\]
and thus, for any $\kappa \in (0,\frac12)$, we obtain
\begin{equation}			\label{eq1616fri}
\left(\fint_{B_{\kappa r}} \abs{v - (v)_{B_{\kappa r}}}^{\frac12} \right)^{2} \le 2\kappa r \norm{D v}_{L^\infty(B_{r/2})} \le 2 N_0 \kappa \left(\fint_{B_r} \abs{v - c}^{\frac12} \right)^{2},
\end{equation}
where $N_0=N_0(d, \delta)$.
We set
\begin{equation}				\label{eq1100sat}
\varphi(x_0,r):=\inf_{c\in \bR}\left( \fint_{B_r(x_0)} \abs{u - c}^{\frac12} \,\right)^{2}.
\end{equation}

Since $u=v+w$, we obtain from the quasi triangle inequality, \eqref{eq1622thu}, and \eqref{eq1616fri}  that
\begin{align*}
\varphi(x_0,\kappa r) & \le \left(\fint_{B_{\kappa r}} \abs{u-(v)_{B_{\kappa r}}}^{\frac12}\right)^{2}
\le 2\left(\fint_{B_{\kappa r}} \abs{v-(v)_{B_{\kappa r}}}^{\frac12}\right)^{2} + 2 \left(\fint_{B_{\kappa r}} \abs{w}^{\frac12}\right)^{2}\\
& \le 4N_0 \kappa \left(\fint_{B_r} \abs{v-c}^{\frac12}\right)^{2} + N \kappa^{-2d}\left(\fint_{B_{r}} \abs{w}^{\frac12}\right)^{2}\\	
& \le 4N_0 \kappa \left(\fint_{B_r} \abs{u-c}^{\frac12}\right)^{2} + N \left(\kappa+\kappa^{-2d}\right)\left(\fint_{B_{r}} \abs{w}^{\frac12}\right)^{2}\\
& \le 4N_0 \kappa \left(\fint_{B_r} \abs{u-c}^{\frac12}\right)^{2} + N \left(\kappa+\kappa^{-2d}\right) \left(\omega_{\rm coef}(r) \,\norm{u}_{L_\infty(B_r)} + \omega_{\rm dat}(r)\right).
\end{align*}
Let $\beta \in (0,1)$ be an arbitrary but fixed number.
With a given $\beta$, we choose a number $\kappa=\kappa(d, \delta, \beta) \in (0, \frac12)$ such that $4N_0 \kappa \le  \kappa^{\beta}$.
Then, we get
\[
\phi(x_0, \kappa r) \le \kappa^\beta \phi(x_0, r) +N \norm{u}_{L_\infty(B_r(x_0))}\,\omega_{\rm coef}(r) + N \omega_{\rm dat}(r).
\]
Note that \eqref{eq1100sat} and H\"older's inequality implies
\[
\varphi(x_0,r) \le \frac{1}{r} \fint_{B_{r}(x_0)} \abs{u}.
\]
Moreover, by replicating the same argument as in \cite{DK17}, we obtain that
\[
\norm{u}_{L_\infty(B_2)} \le N \norm{u}_{L_1(B_4)} + N \int_0^{1} \frac{\tilde\omega_{\rm dat}(t)}t \,dt,
\]
where $N$ depends only on $d$, $\delta$, $\omega_{\rm coef}$.
Refer to \cite[p. 461]{DEK18} for the definition of $\tilde\omega_{\rm dat}(t)$.
Furthermore, we have
\begin{multline*}
\abs{u(x)-u(y)} \lesssim \norm{u}_{L_1(B_4)} \,\abs{x-y}^\beta\\
+\left(\norm{u}_{L_1(B_4)} + \int_0^{1} \frac{\tilde\omega_{\rm dat}(t)}t \,dt\right) \int_0^{\abs{x-y}} \frac{\tilde \omega_{\rm coef}(t)}t \,dt + \int_0^{\abs{x-y}} \frac{\tilde \omega_{\rm dat}(t)}t \,dt,
\end{multline*}
where $N$ depends only on $d$, $\delta$, $\beta$, $\omega_{\rm coef}$.
It is routine to verify that a similar estimates are also available with balls in \eqref{eq10.23m} and \eqref{eq14.00} are replaced by half balls.
By employing the boundary flattening method as in \cite{DEK18}, and using the partition of unity, we conclude that 
\begin{equation}					\label{eq10.23m}
\norm{u}_{L_\infty(\Omega)} \le N \left( \norm{u}_{L_1(\Omega)} + 1\right),
\end{equation}
where $N$ depends only on $d$, $\delta$, $\omega_{\rm coef}$, $\omega_{\rm dat}$, and $\Omega$.
Moreover, there exists moduli of continuity $\varrho_{\rm coef}$ and $\varrho_{\rm dat}$, determined by $\omega_{\rm coef}$ and $\omega_{\rm dat}$, respectively, such that we have
\begin{equation}				\label{eq14.00}
\abs{u(x)-u(y)} \le N \left( \norm{u}_{L_1(\Omega)} + 1\right)\varrho_{\rm coef}(\abs{x-y})+ N \varrho_{\rm dat}(\abs{x-y}),
\end{equation}
where $N$ is a constant depending only on $d$, $\delta$, $\omega_{\rm coef}$, $\omega_{\rm dat}$, and $\Omega$.

Now, we drop the assumption that $\boldsymbol{b}$ and  $c$ are bounded functions.
Let $\boldsymbol b_n$ and $c_n$ be smooth functions, converging to $\boldsymbol b$ and $c$, respectively, in $L_{1}(\Omega)$.
Additionally, we assume that $c_n \ge 0$.
Note that
\[
\omega_{\rm coef}^{(n)}(r):=  \omega_{\mathbf A}(r)+  r\sup_{x \in \Omega} \fint_{\Omega \cap B_r(x)}\abs{\boldsymbol b_n}+ r^2\sup_{x \in \Omega} \fint_{\Omega \cap B_r(x)}\abs{c_n}
\]
satisfies $\omega_{\rm coef}^{(n)}(r) \le N \omega_{\rm coef}(r)$ uniformly in $n$.

Consider the problem
\begin{equation}			\label{eq1317tue}
\left\{
\begin{aligned}
\dv^2(\mathbf A u_n)-\dv(\boldsymbol b_n u_n)-c_nu&=\dv^2 \mathbf{f} +\dv \boldsymbol{g}+h\;\text{ in }\;\Omega,\\
(\mathbf A \nu\cdot \nu)u&=\mathbf f \nu \cdot \nu\;\text{ on }\;\partial \Omega.
\end{aligned}
\right.
\end{equation}
Let $p \in (1, d/2)$ satisfy the hypothesis of \cite[Theorem 2.8]{Krylov2023e}.
Then, for any $\varphi  \in L_{p}(\Omega)$, there exists a unique $v_n \in \mathring{W}^2_{p}(\Omega)$ such that
\[
\mathrm{tr}(\mathbf A D^2 v_n)+\boldsymbol b_n \cdot  D v_n -c_n v_n=\varphi\;\text{ in }\;\Omega.
\]
Moreover, we have 
\begin{equation}			\label{eq2105mon}
\norm{v_n}_{W^2_p(\Omega)} \le N \norm{\varphi}_{L_p(\Omega)},
\end{equation}
where $N$ is independent of $n$.
By the same argument as in Section \ref{sec2.2}, there exists a unique $u_n \in L_{p'}(\Omega)$ that solves the problem \eqref{eq1317tue}.
It follows from \eqref{eq2105mon} and the H\"older's inequality that
\begin{equation}			\label{eq1321tue}
\norm{u_n}_{L_{1}(\Omega)} \le \abs{\Omega}^{1/p} \norm{u_n}_{L_{p'}(\Omega)} \le N \abs{\Omega}^{1/p}.
\end{equation}
The estimate \eqref{eq1321tue} together with \eqref{eq10.23m} and \eqref{eq14.00} imply that $u_n$ is uniformly bounded and equicontinuous.
Therefore, there is a subsequence of $u_n$ that converges uniformly to a function $u$ satisfying the estimates \eqref{eq10.23m} and \eqref{eq14.00}.
Since each $u_n$ satisfies
\[
\int_{B_4} u_n\left(a^{ij}D_{ij}\eta + b^i_n D_i \eta -c_n\eta\right) =\int_{B_4} f^{ij} D_{ij}\eta - g^i D_i \eta+h \eta,
\]
for any $\eta \in C^\infty(\overline \Omega)$ satisfying $\eta =0$ on $\partial \Omega$, and 
$\boldsymbol b_n \to \boldsymbol b$, and $c_n \to c$  in $L_1(\Omega)$, 
it is straightforward to see that $u$ is a weak solution of the problem
\begin{equation}			\label{eq1540tue}
\left\{
\begin{aligned}
\dv^2(\mathbf A u)-\dv(\boldsymbol b u)-cu&=\dv^2 \mathbf{f} +\dv \boldsymbol{g}+h\;\text{ in }\;\Omega,\\
(\mathbf A \nu\cdot \nu)u&=\mathbf f \nu \cdot \nu\;\text{ on }\;\partial \Omega.
\end{aligned}
\right.
\end{equation}
It is clear that $u$ coincides with the weak solution in $L_{p'}(\Omega)$ of the same problem, the existence of which follows from the same argument as that used for $u_n$.
We have proved the following theorem.

\begin{theorem}
Let $\Omega$ be a bounded $C^{1,1}$ domain in $\mathbb{R}^d$ with $d\ge 3$.
Suppose that $c \ge 0$ and Assumption 2.6 in \cite{Krylov2023e} is satisfied.
Let $\mathbf f \in \mathrm{DMO}(\Omega)$, $\boldsymbol g \in L_{q_0}(\Omega)$, and $h \in L_{q_0/2}(\Omega)$, where $q_0>d$.
Assume that $\omega_{\rm coef}(r)$ as defined in \eqref{omega_coef} satisfies the Dini condition.
Then there exists a unique solution $u \in C(\overline \Omega)$ of the problem \eqref{eq1540tue}.
\end{theorem}

\section{Parabolic equations in double divergence form}			\label{sec4}

We will also consider a parabolic operator in double divergence form given by
\begin{equation*}			
\mathscr{L}^*u:=-\partial_t u + \dv^2 (\mathbf A u) -\dv(\boldsymbol{b} u) = -\partial_t u + D_{ij}(a^{ij} u) - D_i(b^i u).
\end{equation*}
We again note that the operator $\mathscr{L}^*$ is the formal adjoint of $\mathscr{L}$, where
\[
\mathscr{L} v := \partial_t v + a^{ij}D_{ij}v + b^i D_i v.
\]
We assume that the principal coefficient matrix $\mathbf{A} = (a^{ij})$ is symmetric, and there exists a constant $\delta \in (0, 1]$ such that the eigenvalues of $\mathbf{A}(t, x)$ are uniformly bounded in the interval $[\delta, 1/\delta]$ for all $(t, x) \in \mathbb{R}^{d+1}$.

Let $\tilde C_r(t,x)=[t, t+r^2)\times B_r(x)$ and let $\mathbb C_r$ be the collection of all cylinders $\tilde C_r(t,x)$ in $\bR^{d+1}$.
For $C=\tilde C_r(t_0,x_0) \in \mathbb C_r$, we set
\[
\bar{\mathbf A}^{x}_{C}(t)=\fint_{B_r(x_0)} \mathbf A(t,x)\,dx,
\]
and for $r>0$, we define
\[
\omega_{\mathbf A}^{x}(r):=\sup_{C \in \mathbb{C}_r}\fint_{C} \,\Abs{\mathbf A(t,x)- \bar{\mathbf A}^{x}_{C}(t)}\,dxdt.
\]
We assume that $\mathbf{A}$ is of Dini mean oscillation with respect to the spatial variable $x$, 
characterized by:
\[
\int_0^{1} \frac{\omega_{\mathbf A}^{x}(t)}{t}\,dt <\infty.
\]

We denote $\mathbf{A} \in \mathrm{DMO}_x$ if $\mathbf{A}$ is of Dini mean oscillation with respect to $x$. It is worth noting that if $\mathbf{A}$ belongs to $\mathrm{DMO}_x$, then it also belongs to the space $\mathrm{VMO}_x$. For the definition of $\mathrm{VMO}_x$, refer to \cite{Krylov2007}.
We allow the coefficient $\boldsymbol b$ and $c$ to be singular.

\subsection{Definition of weak solutions}
For a domain $\Omega$ in $\mathbb{R}^{d}$, and numbers $-\infty \le t_0 <t_1 \le \infty$, we denote
\begin{equation}			\label{eq0819thu}
\Omega_{t_0,t_1}:=(t_0,t_1) \times \Omega,\quad \partial_x\Omega_{t_0,t_1}:=(t_0,t_1) \times \partial\Omega,\quad \partial_{x,t} \Omega_{t_0,t_1}:= \partial_x\Omega_{t_0,t_1} \cup \{t_1\}\times \Omega.
\end{equation}
\begin{definition}
Let $\Omega$ be a bounded $C^{1,1}$ domain in $\mathbb{R}^d$.
Let $\mathbf f=(f^{ij})$, $\boldsymbol g=(g^1,\ldots, g^d)$, $h \in L_{1}(\Omega_{t_0,t_1})$, and $\psi \in L_{1}(\partial_x \Omega_{t_0,t_1})$.
We say that $u\in L_1(\Omega_{t_0,t_1})$ is a weak solution of
\begin{align*}
&\mathscr{L}^*u-cu=\dv^2 \mathbf{f}+\dv \boldsymbol g+h \;\text{ in }\;(t_0,t_1)\times \Omega,\\
&u=\frac{\mathbf f \nu \cdot \nu}{\mathbf A \nu\cdot \nu}+\psi \;\text{ on }\;(t_0,t_1) \times \partial\Omega,\quad
u(t_0,\cdot)=0\;\text{ on }\;\Omega,
\end{align*}
if $\boldsymbol b u \in L_1(\Omega_{t_0,t_1})$, $cu \in L_1(\Omega_{t_0,t_1})$, and for every $\eta \in C^\infty(\overline \Omega_{t_0,t_1})$ satisfying $\eta=0$ on $\partial_{x,t} \Omega_{t_0,t_1}$, we have
\begin{multline}				\label{eq1029thu}
\int_{t_0}^{t_1}\!\!\!\int_{\Omega} u\left(\partial_t \eta+ a^{ij}D_{ij}\eta + b^i D_i \eta -c\eta\right) \\
=\int_{t_0}^{t_1}\!\!\!\int_{\Omega} f^{ij}D_{ij}\eta - g^i D_i \eta+h \eta+\int_{t_0}^{t_1}\!\!\!\int_{\partial\Omega} \psi a^{ij}D_j\eta \nu_i.
\end{multline}
Let $Q$ be a domain in $\bR^{d+1}$.
Let $\mathbf f$, $\boldsymbol g$, $h \in L_{1,\rm{loc}}(Q)$.
We say that $u \in L_{1,\rm{loc}}(Q)$ is a weak solution of
\[
\mathscr{L}^*u-cu=\dv^2 \mathbf f+\dv \boldsymbol g +h\;\text{ in }\;Q,
\]
if $\boldsymbol b u \in L_{1,\rm{loc}}(Q)$, $cu \in L_{1,\rm{loc}}(Q)$, and for $\eta \in C^\infty_c(Q)$, we have
\[
\int_{Q} u\left(\partial_t \eta+ a^{ij}D_{ij}\eta + b^i D_i \eta -c\eta\right) =\int_{Q} f^{ij}D_{ij}\eta - g^i D_i \eta+h \eta.
\]
\end{definition}

\subsection{Existence of weak solutions}
We are looking for a weak solution $u \in L_{q_0}(\Omega_{t_0,t_1})$ satisfying \eqref{eq1029thu}, where $q_0 >(d+2)/2$.
We assume that $\boldsymbol b \in L_{p_0}(\Omega_{t_0,t_1})$, $c \in L_{p_0/2}(\Omega_{t_0,t_1})$ with $p_0>d+2$. 
We will utilize the unique solvability in  $\mathring W^{1,2}_{q_0'}(\Omega_{t_0,t_1})$ of the problem
\[
\mathscr{L} v - cv =\varphi \;\text{ in }\;\Omega_{t_0,t_1},\quad v=0\;\text{ on }\;\partial_{x,t}\Omega_{t_0,t_1}.
\]
Here, $\mathring W^{1,2}_{p}(\Omega)$ denotes the completion of the set $\Set{u \in C^\infty(\overline \Omega_{t_0,t_1}): u=0 \text{ on }\partial_{x,t}\Omega_{t_0,t_1}}$ in  the norm of $W^{1,2}_p$.
For $q_0' \in (1,(d+2)/2)$, let $\mathbf f \in L_{q_0}(\Omega_{t_0,t_1})$, $\boldsymbol g \in L_{(d+2)q_0/(d+2+q_0)}(\Omega_{t_0,t_1})$,  $h \in L_{(d+2)q_0/(d+2+2q_0)}(\Omega_{t_0,t_1})$, and $\psi \in L_{q_0}(\partial\Omega)$.

Consider the mapping $T:L_{q_0'}(\Omega_{t_0,t_1})\to \mathbb{R}$ given by
\[
T(\varphi)=\int_{\Omega_{t_0,t_1}} f^{ij}D_{ij} v- \int_{\Omega_{t_0,t_1}}a g^i D_i v + \int_{\Omega_{t_0,t_1}} hv+\int_{\partial_x\Omega_{t_0,t_1}} \psi a^{ij}D_jv\nu_i,
\]
where $v \in \mathring W^{1,2}_{q_0'}(\Omega)$ is a strong solution of $\mathscr{L}v -c v =\varphi$.
We observe that 
\begin{align*}
\Abs{\int_{\Omega_{t_0,t_1}} f^{ij}D_{ij} v\,} & \le \norm{\mathbf f}_{L_{q_0}(\Omega_{t_0,t_1})} \norm{D^2 v}_{L_{q_0'}(\Omega_{t_0,t_1})}\le 
N \norm{\mathbf f}_{L_{q_0}(\Omega_{t_0,t_1})} \norm{\varphi}_{L_{q_0'}(\Omega_{t_0,t_1})},\\
\Abs{\int_{\Omega_{t_0,t_1}} g^{i}D_{i} v\,} &\le \norm{\boldsymbol g}_{L_{(d+2)q_0/(d+2+q_0)}(\Omega_{t_0,t_1})} \norm{D v}_{L_{q_0'(d+2)/(d+2-q_0')}(\Omega_{t_0,t_1})}\\
&\le N \norm{\boldsymbol g}_{L_{(d+2)q_0/(d+2+q_0)}(\Omega_{t_0,t_1})} \norm{\varphi}_{L_{q_0'}(\Omega_{t_0,t_1})},\\
\Abs{\int_{\Omega_{t_0,t_1}} h v\,} &\le \norm{h}_{L_{(d+2)q_0/(d+2+2q_0)}(\Omega_{t_0,t_1})} \norm{v}_{L_{q_0'(d+2)/(d+2-2q_0')}(\Omega_{t_0,t_1})}\\
&\le N \norm{h}_{L_{(d+2)q_0/(d+2+2q_0)}(\Omega_{t_0,t_1})} \norm{\varphi}_{L_{q_0'}(\Omega_{t_0,t_1})},\\
\Abs{\int_{\partial_x\Omega_{t_0,t_1}} \psi  a^{ij}D_jv \nu_i\,} &\le N  \norm{\psi}_{L_{q_0}(\partial_x\Omega_{t_0,t_1})} \norm{Dv}_{W^{1,1}_{q_0'}(\Omega)}\le N \norm{\psi}_{L_{q_0}(\partial_x\Omega_{t_0,t_1})} \norm{\varphi}_{L_{q_0'}(\Omega_{t_0,t_1})}.
\end{align*}
Here, we used the anisotropic Sobolev embedding theorem and the trace theorem.
Therefore, $T$ is a bounded functional on $L_{q_0'}(\Omega_{t_0,t_1})$, and by the Riesz representation theorem, there is unique $u \in L_{q_0}(\Omega_{t_0,t_1})$ such that
\[
T(\varphi)=\int_{\Omega_{t_0,t_1}} u \varphi,\quad \forall \varphi \in L_{q_0'}(\Omega_{t_0,t_1}).
\]
Moreover, we have
\begin{multline}			\label{eq1136wed}
\norm{u}_{L_{q_0}(\Omega_{t_0,t_1})}
 \lesssim \norm{\mathbf f}_{L_{q_0}(\Omega_{t_0,t_1})} + \norm{\boldsymbol g}_{L_{(d+2)q_0/(d+2+q_0)}(\Omega_{t_0,t_1})}\\
+ \norm{h}_{L_{(d+2)q_0/(d+2+2q_0)}(\Omega_{t_0,t_1})}+\norm{\psi}_{L_{q_0}(\partial_x\Omega_{t_0,t_1})}.
\end{multline}
We have proved the following theorem.
\begin{theorem}
Let $\Omega$ be a bounded $C^{1,1}$ domain in $\mathbb{R}^d$, and recall notations in \eqref{eq0819thu}.
Let $\mathbf A \in \mathrm{VMO}_x$, $\boldsymbol b \in L_{p_0}(\Omega_{t_0,t_1})$, and $c \in L_{p_0/2}(\Omega_{t_0,t_1})$ with $p_0>d+2$.
For $q_0 >(d+2)/d$, let $\mathbf f \in L_{q_0}(\Omega_{t_0,t_1})$, $\boldsymbol g \in L_{(d+2)q_0/(d+2+q_0)}(\Omega_{t_0,t_1})$, $h \in L_{(d+2)q_0/(d+2+2q_0)}(\Omega_{t_0,t_1})$, and $\psi \in L_{q_0}(\partial\Omega)$.
There exists a unique weak solution $u \in L_{q_0}(\Omega_{t_0,t_1})$ of the problem
\begin{align*}
&\mathscr{L}^*u-cu=\dv^2 \mathbf{f}+\dv \boldsymbol g+h \;\text{ in }\;\Omega_{t_0,t_1},\\
&u=\frac{\mathbf f \nu \cdot \nu}{\mathbf A \nu\cdot \nu}+\psi \;\text{ on }\; \partial_x \Omega_{t_0,t_1},\quad
u(t_0,\cdot)=0\;\text{ on }\;\Omega,
\end{align*}
satisfying the estimate \eqref{eq1136wed}.
Moreover, the identity \eqref{eq1029thu} holds with any $\eta \in \mathring W^{1,2}_{q_0'}(\Omega)$.
\end{theorem}

\subsection{Higher integrability of weak solutions}
For $X_0=(t_0, x_0) \in \bR^{d+1}$ and $r>0$ we define the parabolic cylinders
\[
Q_r(X_0)=(t_0-r^2,t_0+r^2)\times B_r(x_0).
\]
Let $u$ be a weak solution of 
\[
\mathscr{L}^*u-cu=\dv^2 \mathbf f+\dv \boldsymbol g +h\;\text{ in }\;Q_2,
\]
where $\mathbf f \in L_{q_0}(\Omega_{t_0,t_1})$, $\boldsymbol g \in L_{(d+2)q_0/(d+2+q_0)}(\Omega_{t_0,t_1})$,  and $h \in L_{(d+2)q_0/(d+2+2q_0)}(\Omega_{t_0,t_1})$ for some $q_0' \in (1,(d+2)/2)$.
By definition of a weak solution, we have
\begin{equation}				\label{eq1602wed}
\int_{Q_2} u \left(\partial_t \eta +a^{ij} D_{ij} \eta+b^i D_i \eta -c\eta \right)=\int_{Q_2} f^{ij} D_{ij}\eta+g^i D_i \eta -h\eta,\quad \forall \eta \in C^{\infty}_c(Q_2).
\end{equation}
For any $\varphi \in C^{\infty}_c(Q_2)$, let $v \in W^{1,2}_p(Q_2)$ be the solution of the problem
\begin{equation}				\label{eq1603wed}
v_t +a^{ij}D_{ij} v = \varphi \;\text{ in }\;Q_2,
\end{equation}
with zero boundary condition on the parabolic boundary $\partial_{x,t}Q_2$.
By \cite[Theorem 3.2]{DEK21}, we have $v \in W^{1,2}_\infty(Q_{r})$ for any $r \in (0,2)$.
Let $v_{\varepsilon}$ be the standard mollification of $v$.
For any fixed cut-off function $\zeta \in C^\infty_c(Q_{r})$ such that $\zeta=1$ on $Q_{\rho}$, where $1<\rho<r<2$, we take $\eta=\zeta v_{\varepsilon} \in C^{\infty}_c(Q_2)$.
Then by \eqref{eq1602wed}, we have
\begin{multline}			\label{eq1450thu}
\int_{Q_2} u (\partial_t\zeta v_{\varepsilon} + \zeta \partial_t v_{\varepsilon}) +ua^{ij} (D_{ij} \zeta v_{\varepsilon} + 2 D_i\zeta D_j v_{\varepsilon} + \zeta D_{ij}v_{\varepsilon})+ ub^i (D_{i} \zeta v_{\varepsilon}+\zeta D_i v_{\varepsilon})-cu\zeta v_{\varepsilon} \\
=\int_{Q_2} f^{ij}(D_{ij} \zeta v_{\varepsilon} + 2 D_i\zeta D_j v_{\varepsilon} 
+ \zeta D_{ij}v_{\varepsilon})+g^i (D_{i} \zeta v_{\varepsilon}+\zeta D_i v_{\varepsilon}) +h \zeta v_{\varepsilon}.
\end{multline}
There is a sequence $\varepsilon_n$ converging to $0$ such that
\[
v_{\varepsilon_n} \to v, \quad D v_{\varepsilon_n} \to Dv,\quad D^2 v_{\varepsilon_n} \to D^2 v, \quad \partial_t v_{\varepsilon_n} \to \partial_t v\quad\text{a.e. in }\;Q_r.
\]
Also, note that for all $\varepsilon \in (0,\frac{2-r}{2})$ we have
\[
\norm{v_{\varepsilon}}_{W^{1,2}_\infty(Q_{r})} \le \norm{v}_{W^{1,2}_\infty(Q_{(2+r)/2})}<\infty.
\]
Hence, by using the dominated convergence theorem, we find that \eqref{eq1450thu} is valid with $v_{\varepsilon}$ replaced by $v$, that is,
\begin{multline}			\label{eq1148sat}
\int_{Q_2} u (\partial_t\zeta v+ \zeta \partial_t v) +ua^{ij} (D_{ij} \zeta v + 2 D_i\zeta D_j v + \zeta D_{ij}v)+ ub^i (D_{i} \zeta v+\zeta D_i v)-cu\zeta v\\
=\int_{Q_2} f^{ij}(D_{ij} \zeta v+ 2 D_i\zeta D_j v
+ \zeta D_{ij}v)+g^i (D_{i} \zeta v+\zeta D_i v) +h \zeta v.
\end{multline}

On the other hand, by multiplying $\zeta u$ to \eqref{eq1603wed} and noting that $v_t$, $D^2 v$, $\varphi \in L_\infty(Q_r)$ and $\zeta u \in L_1(Q_2)$, we have
\begin{equation}			\label{eq1627thu}
\int_{Q_2} v_t \zeta u + a^{ij} D_{ij}v \zeta u  = \int_{Q_2} \varphi \zeta u.
\end{equation}
By combing \eqref{eq1148sat} and \eqref{eq1627thu}, we obtain
\begin{multline}				\label{eq1652sat}
\int_{Q_2} \zeta u \varphi =  \int_{Q_2} f^{ij}(D_{ij} \zeta v+ 2 D_i\zeta D_j v+ \zeta D_{ij}v)+g^i (D_{i} \zeta v+\zeta D_i v) +h \zeta v\\
-\int_{Q_2} u(\partial_t\zeta v + a^{ij} D_{ij}\zeta v+ 2a^{ij} D_i\zeta D_jv) + b^i u(D_i\zeta v+\zeta D_i v)-cu \zeta v.
\end{multline}
Note that we can choose $\zeta$ such that
\[
\norm{\zeta}_\infty \le1,\quad \norm{D \zeta}_\infty  \le 2/(r-\rho),\quad \norm{\partial_t \zeta}_\infty+\norm{D^2 \zeta}_\infty \le 8/(r-\rho)^2.
\]
Take $r=r_0$ and $\rho=r_1$, where $1<r_0<r_1<2$.
It follows from \eqref{eq1148sat} that
\begin{multline}			\label{eq1435sat}
\Abs{\int_{Q_2} \zeta u \varphi\,} \lesssim 
\left(\norm{\mathbf f}_{L_{q_0}(Q_2)} +\norm{\boldsymbol g}_{L_{(d+2)q_0/(d+2+q_0)}(Q_2)} +\norm{h}_{L_{(d+2)q_0/(d+2+2q_0)}(Q_2)}\right)\norm{v}_{W^{1,2}_{q_0'}(Q_2)}\\
+\norm{u}_{L_1(Q_2)}\norm{v}_{W^{0,1}_\infty(Q_2)} + \norm{\boldsymbol b u}_{L_1(Q_2)} \norm{v}_{W^{0,1}_\infty(Q_2)}+\norm{cu}_{L_1(Q_2)} \norm{v}_{L_\infty(Q_2)}.
\end{multline}
where $A \lesssim B$ means $A \le N(r_1-r_0)^{-2} B$ and $\norm{v}_{W^{0,1}_\infty(Q_2)} :=\norm{v}_{L_\infty(Q_2)} +\norm{Dv}_{L_\infty(Q_2)}$.

Since $a^{ij} \in \mathrm{DMO}_x \subset \mathrm{VMO}_x$, we can apply the standard $L_p$ theory, that is, for any $p \in (1,\infty)$, we have
\begin{equation}			\label{eq1508wed}
\norm{v}_{W^{1,2}_p(Q_2)} \le N \norm{\varphi}_{L_p(Q_2)}.
\end{equation}
Then it follows from the Sobolev embedding (see, for instance, \cite{LSU}) that for any $p_1' \in (d+2,\infty)$, we have
\begin{equation}			\label{eq0944sat}
\norm{v}_{L_\infty(Q_2)}+ \norm{Dv}_{L_\infty(Q_2)} \le N\norm{v}_{W^{1,2}_{p_1'}(Q_2)}\le  N \norm{\varphi}_{L_{p_1'}(Q_2)}.
\end{equation}
Note that $q_0'<p_1'$.
Since $g \in C^\infty_c(Q_2)$ is arbitrary, by utllizing \eqref{eq1435sat}, \eqref{eq0944sat}, and applying the converse of H\"older's inequality, we have
\begin{multline*}
\norm{u}_{L_{p_1}(Q_{r_1})}\le \norm{\zeta u}_{L_{p_1}(Q_2)}
\lesssim \norm{\mathbf f}_{L_{q_0}(Q_2)} +\norm{\boldsymbol g}_{L_{(d+2)q_0/(d+2+q_0)}(Q_2)} +\norm{h}_{L_{(d+2)q_0/(d+2+2q_0)}(Q_2)}\\
+\norm{u}_{L_1(Q_2)}+\norm{\boldsymbol b u}_{L_1(Q_2)}+ \norm{cu}_{L_1(Q_2)},\quad \forall p_1 \in (1, (d+2)/(d+1)).
\end{multline*}

We can iterate this process.
Let $k$ be the largest integer such that
\[
1/(d+2)+k(1/(d+2)-1/p_0)\le 1/q_0'.
\]
We take $p_1'>d+2$ sufficiently close to $d+2$ such that
\[
1/p_1'+k(1/(d+2)-1/p_0)<1/q_0'\quad\text{and}\quad 1/p_1'+(k+1)(1/(d+2)-1/p_0)>1/q_0'.
\]
For $j=1,\ldots,k$, we set
\[
1/p_{j}'=1/p_1'+(j-1)(1/(d+2)-1/p_0),\quad 1/s_{j}=1/p_{j}'-1/p_0,\quad 1/s_{j}^*=1/s_{j}-1/(d+2).
\]
Taking $r_{j+1}$ and $r_j$ in the place of $\rho$ and $r$, and going back to \eqref{eq1652sat}, we see that
\begin{multline}			\label{eq1644sat}	
\Abs{\int_{Q_2} \zeta u \varphi\,} \lesssim
\left(\norm{\mathbf f}_{L_{q_0}(Q_2)} +\norm{\boldsymbol g}_{L_{(d+2)q_0/(d+2+q_0)}(Q_2)} +\norm{h}_{L_{(d+2)q_0/(d+2+2q_0)}(Q_2)}\right)\norm{v}_{W^{1,2}_{q_0'}(Q_2)}\\
+\norm{u}_{L_{p_{j}}(Q_{r_j})} \left(\norm{v}_{W^{0,1}_{s_j}(Q_2)}+\norm{\boldsymbol b}_{L_{p_0}(Q_2)} \norm{v}_{W^{0,1}_{s_j}(Q_2)} +\norm{c}_{L_{p_0/2}(Q_2)} \norm{v}_{L_{s_j^*}(Q_2)} \right),
\end{multline}
where we used
\[
1/p_j+1/p_0+1/s_j=1,\quad1/p_j+2/p_0+1/s_j^*<1.
\]
Note that for $j=1,\ldots, k$, we have
\[
1/p_{j+1}'=1/s_j+1/(d+2)\quad\text{and}\quad 1/p_{j+1}'<1/q_0'.
\]
Hence, it follows from \eqref{eq1508wed} and the anisotropic Sobelev inequality that
\[
\norm{v}_{L_{s_j^*}(Q_2)}+ \norm{Dv}_{L_{s_j}(Q_2)} \le N \norm{\varphi}_{L_{p_{j+1}'}(Q_2)},\quad
\norm{v}_{W^{1,2}_{q_0'}(Q_2)} \le  N \norm{\varphi}_{L_{p_{j+1}'}(Q_2)}.
\]
Therefore, we derive from \eqref{eq1644sat} and the converse of H\"older's inequality that
\begin{multline*}
\norm{u}_{L_{p_{j+1}}(Q_{r_{j+1}})}\le \norm{\zeta u}_{L_{p_{j+1}}(Q_2)} \lesssim \norm{\mathbf f}_{L_{q_0}(Q_2)}+\norm{\boldsymbol g}_{L_{(d+2)q_0/(d+2+q_0)}(Q_2)}+\norm{h}_{L_{(d+2)q_0/(d+2+2q_0)}(Q_2)}\\
+\norm{u}_{L_{p_j}(Q_{r_j})}\left(1+ \norm{\boldsymbol b}_{L_{p_0}(Q_2)}+\norm{c}_{L_{p_0/2}(Q_2)}\right),\quad j=1,\ldots,k.
\end{multline*}
Here, we use the notation $A \lesssim B$ to mean that $A \le N(r_{j+1}-r_j)^{-2} B$.

In the case when $j=k+1$, we set
\[
1/p_{k+2}'=1/q_0',\quad 1/s_{k+1}=1/q_0'-1/(d+2), \quad  1/s_{k+1}^*=1/s_{k+1}-1/(d+2).
\]
Taking $r=r_{k+1}$ and $\rho=r_{k+2}$, similar to \eqref{eq0957tue}, we have
\begin{multline*}	
\Abs{\int_{Q_2} \zeta u \varphi\,} \lesssim
\left(\norm{\mathbf f}_{L_{q_0}(Q_2)} +\norm{\boldsymbol g}_{L_{(d+2)q_0/(d+2+q_0)}(Q_2)} +\norm{h}_{L_{(d+2)q_0/(d+2+2q_0)}(Q_2)}\right)\norm{v}_{W^{1,2}_{q_0'}(Q_2)}\\
+\norm{u}_{L_{p_{k+1}}(Q_{r_{k+1}})} \left(\norm{v}_{W^{0,1}_{s_{k+1}}(Q_2)}+\norm{\boldsymbol b}_{L_{p_0}(Q_2)} \norm{v}_{W^{0,1}_{s_{k+1}}(Q_2)} +\norm{c}_{L_{p_0/2}(Q_2)} \norm{v}_{L_{s_{k+1}^*}(Q_2)} \right),
\end{multline*}
where we used 
\[
1/p_{k+1}+1/p_0+1/s_{k+1}<1,\quad 1/p_{k+1}+2/p_0+1/s_{k+1}^*<1.
\]
Noting $\norm{v}_{W^{0,1}_{s_{k+1}}(Q_2)} \le N \norm{\varphi}_{L_{q_{0}'}(Q_2)}$, and invoking the converse of H\"older's inequality, we obtain
\begin{multline*}
\norm{u}_{L_{q_{0}}(Q_{r_{k+2}})}\le \norm{\zeta u}_{L_{q_{0}}(Q_2)} \lesssim \norm{\mathbf f}_{L_{q_0}(Q_2)}+\norm{\boldsymbol g}_{L_{(d+2)q_0/(d+2+q_0)}(Q_2)}+\norm{h}_{L_{(d+2)q_0/(d+2+2q_0)}(Q_2)}\\
+\norm{u}_{L_{p_{k+1}}(Q_{r_{k+1}})}\left(1+ \norm{\boldsymbol b}_{L_{p_0}(Q_2)}+\norm{c}_{L_{p_0/2}(Q_2)}\right).
\end{multline*}

We have proved the following theorem.
\begin{theorem}
Let $\mathbf A \in \mathrm{DMO}_x$, $\boldsymbol b \in L_{p_0}(Q_{r})$, $c \in L_{p_0/2}(Q_{r})$ with $p_0>d+2$.
Let $\mathbf f \in L_{q_0}(Q_{r})$, $\boldsymbol g \in L_{(d+2)q_0/(d+2+q_0)}(Q_{r})$, and $h \in L_{(d+2)q_0/(d+2+2q_0)}(Q_{r})$, where $q_0>(d+2)/d$.
Suppose $u \in L_{1}(Q_{r})$ is a weak solution of
\[
\mathscr{L}^*u-cu=\dv^2 \mathbf{f}+\dv \boldsymbol g+h\;\text{ in }\;Q_{r}.
\]
Then we have $u \in L_{q_0}(Q_{\rho})$ for any $\rho<r$.
\end{theorem}

\subsection{Continuity}
The following boundary continuity result is obtained by adapting the arguments of \cite[Theorem 1.4]{DEK21}.

\begin{theorem}
Let $\Omega$ be a bounded $C^{1,1}$ domain in $\mathbb{R}^d$, and recall notations in \eqref{eq0819thu}.
Assume $\mathbf A \in \mathrm{DMO}_x$, $\boldsymbol b \in L_{p_0}(\Omega_{t_0,t_1})$, and $c \in L_{p_0/2}(\Omega_{t_0,t_1})$ with $p_0>d+2$.
Suppose $u \in L_1(\Omega_{t_0,t_1})$ satisfies
\[
\mathscr{L}^*u-cu=\dv^2 \mathbf{f}+\dv \boldsymbol g+h \;\text{ in }\; \Omega_{t_0,t_1},\quad 
u=\frac{\mathbf f \nu \cdot \nu}{\mathbf A \nu\cdot \nu} \;\text{ on }\;\partial_{x}\Omega_{t_0,t_1},
\]
where $\mathbf{f} \in \mathrm{DMO}_x \cap L_\infty(\Omega_{t_0,t_1})$,  $\boldsymbol g\in L_{q_0}(\Omega_{t_0,t_1})$, and $h\in L_{q_0/2}(\Omega_{t_0,t_1})$ for some $q_0>d+2$.
Assume additionally that $\mathbf f \nu \cdot \nu$ vanishes near $\partial_x \Omega_{t_0,t_1}$.
Then, for any $t_0'$ and $t_1'$ satisfying $t_0<t_0'<t_1'<t_1$, we have $u\in C^{0,0}(\overline \Omega_{t_0',t_1'})$.
\end{theorem}

\subsection{Parabolic Harnack inequality}			\label{sec4.2}
The following theorem is contained in \cite[Theorem 1.4]{GK24}, which covers more general cases.
For further details, refer to \cite{GK24}.
Here, we assume that $\mathbf{A}$, $\boldsymbol{b}$, and $c$ are defined over the entire space $\mathbb{R}^{d+1}$.

\begin{theorem}
Let $\mathbf A \in \mathrm{DMO}_x$, $\boldsymbol b \in L_{p_0}$, and $c \in L_{p_0/2}$ for some $p_0> d+2$.
Let $R>0$ be a fixed number, $0<r<R/4$, and $(t_0, x_0) \in \bR^{d+1}$.
Denote $C_{r}=(t_0-r^2, t_0] \times B_{r}(x_0)$. 
Suppose $u \in L_1(C_{4r})$ is a nonnegative solution of
\[
\mathscr{L}^* u - cu =0\;\;\text{ in }\;\; C_{4r}.
\]
Then, we have
\[
\sup_{(t_0-3r^2, t_0-2r^2)\times B_{r/3}(x_0)} u \le N \inf_{(t_0-r^2, t_0)\times B_{r/3}(x_0)} u,
\]
where $N$ is a constant depending only on $d$, $\delta$, $\omega_{\mathbf A}^{x}$, $p_0$, $\norm{\boldsymbol b}_{L_{p_0}}$, $\norm{c}_{L_{p_0/2}}$, and $R$.
\end{theorem}


\end{document}